\documentclass[12pt]{amsart}
\usepackage[utf8]{inputenc}
\usepackage[english]{babel}
\usepackage{amsmath, amssymb, amsthm, amscd,color,comment}
\usepackage{cancel}
\usepackage{cite}
\usepackage{alltt}
\usepackage[dvipsnames]{xcolor}
\usepackage{array}
\usepackage[small,bf,labelsep=period]{caption}
\usepackage{mathtools}
\usepackage{textcomp}
\usepackage{enumerate}
\newtheorem{theorem}{Theorem}[section]

\newtheorem{remark}[theorem]{Remark}
\newtheorem{lemma}[theorem]{Lemma}

\newtheorem{corollary}[theorem]{Corollary}
\newtheorem{proposition}[theorem]{Proposition}

\begin{document}
	\setlength{\parindent}{0cm}	
	\title[Ruelle's inequality and the Pesin formula]{Ruelle's inequality and the Pesin formula for geodesic flows on manifolds without conjugate points}
	
	
	\author{Alexander Cantoral}
	\address{School of Mathematics, Sun Yat-sen University, Zhuhai, China}
	\email{vilchez@mail.sysu.edu.cn}
	
	\author{Sergio Romaña}
	\address{School of Mathematics, Sun Yat-sen University, Zhuhai, China}
	\email{sergio@mail.sysu.edu.cn}
	\begin{abstract} In this paper, we prove the existence of Lyapunov exponents for the geodesic flow on complete Riemannian manifolds without conjugate points whose sectional curvature is bounded below. Under suitable additional curvature assumptions, we establish Ruelle's inequality. Furthermore, assuming the same geometric conditions, we obtain the Pesin entropy formula for $C^1$-Hölder geodesic flows on finite-volume manifolds without conjugate points.
	\end{abstract}
	
	\maketitle
	\section{Introduction}\noindent
	The classical Ruelle's inequality and Pesin entropy formula are well-established for diffeomorphisms on compact manifolds. In that setting, the existence of Lyapunov exponents is automatically guaranteed, as the norm of the differential remains bounded. The situation becomes considerably more delicate in the non-compact case, where the main difficulty lies in the possible failure of integrability conditions. Despite this obstacle, various extensions of these entropy formulas have been obtained. For instance, Riquelme in \cite{riquelme} proved both Ruelle's inequality and the Pesin entropy formula for geodesic flows on manifolds with pinched negative sectional curvature, under supplementary assumptions on the derivatives of the curvature. More recently, in \cite{cantoralpesin}, we established the same results for Anosov geodesic flows on non-compact manifolds, assuming suitable curvature bounds.\\
	
	In the present work, we go a step further by showing that the Anosov condition is not actually necessary to guarantee the existence of Lyapunov exponents for geodesic flows $\phi$ on manifolds without conjugate points. More precisely, we prove the following integrability result, which lies at the core of our approach:
	\begin{theorem}
		Let $M$ be a complete Riemannian manifold without conjugate points and sectional curvature bounded below by $-c^2$, where $c\ge 0$. Then, for every $\phi^t$-invariant probability measure $\mu$ on $SM$, we have that $\log^+ \left\| d\phi^{\pm 1}\right\|\in L^1(\mu)$.  
	\end{theorem}
	The integrability condition provided by Theorem 1.1 ensures that Oseledets theorem applies. Hence, for $\mu$-almost every $\theta \in SM$, the Lyapunov exponents $\{\lambda_i(\theta)\}_{i=1}^{l(\theta)}$ are well-defined and the tangent space admits an Oseledets splitting $$T_\theta SM = \bigoplus_{i=1}^{l(\theta)} H_i(\theta),$$ where each subspace $H_i(\theta)$ corresponds to the Lyapunov exponent $\lambda_i(\theta)$ with multiplicity $\dim(H_i(\theta))$. Once the existence of Lyapunov exponents has been established for $\mu$-almost every point, we proceed to prove Ruelle's inequality for the geodesic flow on the unit tangent bundle of a non-compact manifold without conjugate points, under suitable curvature assumptions:
	\begin{theorem}
		Let $M$ be a complete Riemannian manifold without conjugate points. Assume that the curvature tensor and the derivative of the curvature tensor are both uniformly bounded. Then, for every $\phi^t$-invariant probability measure $\mu$ on $SM$ we have
		\begin{align*}
			h_\mu(\phi)\le  \int\limits_{SM} \sum_{\lambda_i(\theta)>0}\lambda_i(\theta)\cdot\dim (H_i(\theta))d\mu(\theta).
		\end{align*}   
	\end{theorem}
	Although Theorem 1.1 guarantees the existence of Lyapunov exponents in our setting, the resulting Oseledets decomposition may differ significantly from that arising in the Anosov case considered in \cite{cantoralpesin}. Indeed, for Anosov geodesic flows, at almost every point, the stable and unstable bundles coincide with the direct sums of the Oseledets subspaces associated with the negative and positive Lyapunov exponents, respectively, while the central subspace is one-dimensional and generated by the flow direction. In contrast, for geodesic flows on manifolds without conjugate points and sectional curvature bounded below, additional zero Lyapunov exponents may occur in directions transverse to the flow, so the central Oseledets subspace may have dimension greater than one. Consequently, the uniform hyperbolicity estimates available in the Anosov setting are no longer valid, creating a significant additional difficulty. To overcome this lack of uniformity, we combine the general strategy of \cite{cantoralpesin} with an application of Egorov's theorem. More precisely, Egorov's theorem yields uniform exponential estimates for both the norm and the conorm of the derivative of the flow on compact sets of arbitrarily large measure. These estimates recover the uniform control required in the proof of \cite{cantoralpesin}, allowing us to extend their arguments to the present setting.\\
	
	A key observation is that, in the proof of Theorem 1.2 of \cite{cantoralpesin}, the Anosov assumption is used only to guarantee the existence of Lyapunov exponents, whereas all the essential estimates rely solely on the curvature assumptions. Consequently, combining Theorems 1.1 and 1.2 yields the following extension of the Pesin entropy formula to geodesic flows on finite-volume manifolds without conjugate points:
	\begin{theorem}
		Let $M$ be a complete Riemannian manifold of finite volume, without conjugate points and suppose its geodesic flow is $C^1$-Hölder. Assume that the curvature tensor and the derivative of the curvature tensor are both uniformly bounded. Then, for every $\phi^t$-invariant probability measure $\mu$ on $SM$ which is absolutely continuous relative to the Lebesgue measure, we have
		\begin{align*}
			h_\mu(\phi)= \int\limits_{SM} \sum_{\lambda_i(\theta)>0}\lambda_i(\theta)\cdot\dim (H_i(\theta))d\mu(\theta).
		\end{align*}   
	\end{theorem}
	It is worth emphasizing that our results do not require the geodesic flow to be Anosov. This represents a substantial extension of the classical setting. Indeed, although every Anosov geodesic flow is known to have no conjugate points (see \cite{klin},\cite{mane2},\cite{nocon}), the converse fails in general. There exist many manifolds without conjugate points whose geodesic flow is not Anosov, including flat manifolds and numerous nonpositively curved manifolds. Consequently, our results apply to a considerably broader class of manifolds than those covered by the Anosov framework.

	\subsection*{Structure of the Paper:} 
	In Section 2, we introduce the notation and geometric tools used throughout the paper. In Section 3, we establish the existence of the Oseledets decomposition for the time-one map of the geodesic flow. In Section 4, we prove Ruelle's inequality for geodesic flows on manifolds without conjugate points. Finally, in Section 5, we develop the argument that enables us to derive the Pesin entropy formula by combining our results with those of \cite{cantoralpesin}.
	
\section{Preliminaries and notation}
\noindent
Throughout this paper, $M=(M,g)$ will denote a complete Riemannian manifold of dimension $n\ge 2$ without boundary, $TM$ is the tangent bundle, $SM$ its unit tangent bundle and $\pi:TM\rightarrow M$ will denote the canonical projection, that is, $\pi(x,v)=x$ for $(x,v)\in TM$.
\subsection{Geodesic flow}
Given $\theta=(x,v)\in TM$, we denote by $\gamma_\theta$ the unique geodesic with initial conditions $\gamma_\theta(0)=x$ and $\gamma'_\theta(0)=v$. The geodesic flow is defined as a family of $C^\infty$-diffeomorphisms $\phi^t:TM\rightarrow TM$, where $t\in \mathbb{R}$, and is given by
\begin{align*}
	\phi^t(\theta)=(\gamma_\theta(t), \gamma'_\theta(t)).
\end{align*}
Since geodesics travel with constant speed, the flow $\phi^t$ leaves $SM$ invariant. The geodesic flow generates a vector field $G$ on $TM$ given by
\begin{align*}
	G(\theta)=\left. \dfrac{d}{dt}\right|_{t=0} \phi^t(\theta)=\left. \dfrac{d}{dt}\right|_{t=0}\left( \gamma_\theta(t),\gamma'_\theta(t)\right) . 
\end{align*}
For each $\theta=(x,v)\in TM$, let $V$ be the vertical subbundle of $TM$ whose fiber at $\theta$ is given by $V_\theta=\ker d\pi_\theta$. Let $K: TTM\rightarrow TM$ be the connection map induced by the Riemannian metric (see \cite{paternain}) and denotes by $H$ the horizontal subbundle of $TM$ whose fiber at $\theta$ is given by $H_\theta=\ker K_\theta$. The maps $\left. d \pi_\theta\right|_{H_\theta}:H_\theta\rightarrow T_xM$ and $\left. K_\theta\right|_{V_\theta}: V_\theta\rightarrow T_xM$ are linear isomorphisms. This implies that $T_\theta TM=H_\theta\oplus V_\theta$ and the map $j_\theta:T_\theta TM\rightarrow T_xM\times T_xM$ given by
\begin{align*}
	j_\theta(\xi)=(d\pi_\theta(\xi), K_\theta(\xi))
\end{align*}
is a linear isomorphism. Furthermore, we can identify every element $\xi\in T_\theta TM$ with the pair $j_\theta(\xi)$. Using the decomposition $T_\theta TM=H_\theta \oplus V_\theta$, we endow the tangent bundle $TM$ with a special Riemannian metric that makes $H_\theta$ and $V_\theta$ orthogonal. This metric is called the Sasaki metric and it is given by
\begin{align*}
	\left\langle \xi, \eta \right\rangle_\theta=\left\langle d\pi_\theta(\xi), d\pi_\theta(\eta)\right\rangle_x + \left\langle K_\theta(\xi),K_\theta(\eta)\right\rangle_x . 
\end{align*}
From now on, we shall work with the Sasaki metric restricted to the unit tangent bundle $SM$, and with the geodesic flow restricted to $SM$. For $\theta = (x,v) \in SM$, it follows from the definition that $\xi \in T_\theta SM$ if and only if $\langle K_\theta(\xi), v \rangle = 0$ (see \cite{paternain}). Moreover, $SM$ is complete with respect to the Sasaki metric (see \cite[Lemma 2.1]{cantoralpesin}).	
	
\subsection{Jacobi fields} To study the differential of the geodesic flow using geometric arguments, let us recall the definition of a Jacobi field. A vector field $J$ along a geodesic $\gamma$ in $M$ is called a Jacobi field if it satisfies the Jacobi equation
\begin{align}\label{jacobi}
	J''(t)+R(\gamma'(t),J(t))\gamma'(t)=0,
\end{align}
where $R$ denotes the curvature tensor of $M$ and $"'"$ denotes the covariant derivative along $\gamma$. A Jacobi field is determined by the initial values $J(t_0)$ and $J'(t_0)$, for any given $t_0\in \mathbb{R}$. For every $\theta\in TM$, the map 
\begin{align}\label{iso}
	i_\theta: \xi\rightarrow J_\xi
\end{align}
defines an isomorphism between $T_\theta TM$ and the space of Jacobi fields along $\gamma_\theta$, where $J_\xi(0)=d\pi_\theta(\xi)$ and $J'_\xi(0)=K_\theta(\xi)$. If we denote by $S$ the orthogonal complement of the subspace spanned by $G$, then for every $\theta\in SM$, the map \eqref{iso} defines an isomorphism between $S(\theta)$ and the space of normal Jacobi fields along $\gamma_\theta$. The differential of the geodesic flow is determined by the behavior of the Jacobi fields and, consequently, by the curvature. More precisely, for $\theta\in SM$ and $\xi\in T_\theta SM$ we have (in the horizontal and vertical coordinates)
\begin{align*}
	d\phi^t_\theta(\xi)=(J_\xi(t), J'_\xi(t)), \hspace{0.5cm} t\in \mathbb{R}.
\end{align*}	
	
\subsection{No conjugate points}
Let $\gamma$ be a geodesic joining $p,q\in M$, where $p\neq q$. We say that $p$ and $q$ are conjugate along $\gamma$ if there exists a non-zero Jacobi field along $\gamma$ vanishing at $p$ and $q$. A manifold $M$ has no conjugate points if any pair of points are not conjugate along any geodesic. This is equivalent to the fact that the exponential map is non-singular at every point of $M$.\\

Suppose that $M$ is a complete Riemannian manifold without conjugate points and curvature bounded below by $-c^2$, for some $c>0$. Following the notations of \cite{Knieper}, let $S(t)$ and $U(t)$ denote the stable and unstable Jacobi tensors, respectively. These tensors are well defined and satisfy the following properties:
\begin{itemize}
	\item $S(0) = \operatorname{id}$, and for every $x$, the vector field $J(t)=S(t)x$ is a Jacobi field that remains bounded as $t\to +\infty$.
	\item $U(0) = \operatorname{id}$, and for every $x$, the vector field $J(t)=U(t)x$ is a Jacobi field that remains bounded as $t\to -\infty$.
\end{itemize}
The stable and unstable Jacobi tensors provide a canonical decomposition of the space of Jacobi fields and constitute one of the fundamental tools for studying the asymptotic behavior of geodesics. In particular, they are closely related to the stable and unstable Green bundles and play a central role in the analysis of Lyapunov exponents. The following two results, taken from Knieper's chapter (see \cite{Knieper}), will be essential for establishing the existence of Lyapunov exponents.
\begin{proposition}[Rauch comparison]
	Let $J$ be a normal Jacobi field along a unit-speed geodesic satisfying $J(0)=0$. Assume that the sectional curvature of M is bounded below by $-c^2$, for some $c>0$, and that the geodesic has no conjugate points on $[0,1]$. Then, for every $t\in[0,1]$,
	\begin{equation*}
		\left\|{J(t)}\right\| \le \left\| J'(0)\right\| \frac{\sinh(ct)}{c}.
	\end{equation*}
	Moreover, its covariant derivative satisfies the standard estimate (see \cite[Corollary 2.12]{Knieper}):
	\begin{align*}
		\left\|J'(t)\right\| \le\left\|J'(0)\right\| \cosh(ct).
	\end{align*}
\end{proposition}
\begin{lemma}
	Assume that the curvature tensor satisfies $-c^2\text{id} \le R(t)$ (equivalently, that the sectional curvature is bounded below by $-c^2$) and that the manifold has no conjugate points. Then the stable Jacobi tensor $S(t)$ satisfies
	\begin{align}\label{Knip}
		\|\dot{S}(t)S^{-1}(t)\|_{\text{op}} \le c,
	\end{align}
	for every $t\ge 0$,
\end{lemma}

\begin{remark}
	Inequality \eqref{Knip} implies that, for any stable Jacobi field $J_s(t)=S(t)x$,
	\[
	\|J'_s(t)\| = \|\dot{S}(t)S^{-1}(t) J_s(t)\| \le \|\dot{S}(t)S^{-1}(t)\|_{\text{op}} \|J_s(t)\| \le c \|J_s(t)\|.
	\]
	In particular, for $t=0$, $\|J'_s(0)\| \le c \|J_s(0)\|$.
\end{remark}
Another key ingredient is the following proposition, proved in \cite{cantoralpesin}, which establishes a uniform bound for the differential of the exponential map under suitable curvature assumptions on the manifold. This estimate plays a crucial role in the proofs of both Ruelle's inequality and the Pesin entropy formula. We recall it here for the reader's convenience.
\begin{proposition}[{\cite[Proposition 2.2]{cantoralpesin}}]
	Let $N$ be a complete Riemannian manifold and suppose that the curvature tensor is uniformly bounded. Then, there exists $t_0>0$ such that for every $x\in N$ and every $v,w\in T_xN$ with $\left\| v\right\| =\left\| w\right\| =1$, we have
	\begin{align*}
		\left\| d(exp_x)_{tv}w\right\| \le \dfrac{5}{2}, \hspace{0.5cm} \forall \left| t\right| \le t_0.
	\end{align*}
\end{proposition}

\subsection{Lyapunov exponents}
Let $(M,g)$ be a Riemannian manifold and let $f:M\rightarrow M$ be a $C^1$-diffeomorphism. A point $x$ is said to be (Lyapunov-Perron) regular if there exist numbers $$\left\lbrace \lambda_1(x)<\lambda_2(x)<\ldots<\lambda_{l(x)}(x)\right\rbrace,$$ called Lyapunov exponents, and an invariant splitting $$T_xM=H_1(x)\oplus H_2(x)\oplus \ldots \oplus H_{l(x)}(x),$$ such that for every vector $v\in H_i(x)\setminus \left\lbrace 0\right\rbrace $, we have  
\begin{align*}
	\lim_{n\rightarrow \pm \infty} \dfrac{1}{n}\log \left\|df^n_x(v) \right\|=\lambda_i(x) 
\end{align*}
and
\begin{align*}
	\lim_{n\rightarrow \pm \infty} \dfrac{1}{n}\log\left| \det\left( df^n_x\right) \right|=\sum_{i=1}^{l(x)}  \lambda_i(x)\cdot\dim (H_i(x)). 
\end{align*}
A fundamental property of Lyapunov exponents is that the exponent of every non-zero vector is determined by its decomposition with respect to the Oseledets splitting. More precisely, if $x$ is a regular point and $$v=\sum_{i}v_i \in T_xM, \hspace{0.5cm} v_i\in H_i(x),$$ then the forward Lyapunov exponent of $v$ is the largest exponent among those corresponding to the non-zero components $v_i$.\\ 
Let $\Lambda$ denote the set of regular points. By Oseledets theorem (see \cite{oseledec}), if $\mu$ is an $f$-invariant probability measure on $M$ such that $\log^+ \left\| df^{\pm 1} \right\|$ is $\mu$-integrable, then the set $\Lambda$ has full $\mu$-measure. Moreover, the functions $x\rightarrow \lambda_i(x)$ and $x\rightarrow \dim(H_i(x))$ are $\mu$-measurable and $f$-invariant. In particular, if $\mu$ is ergodic, they are $\mu$-almost everywhere constant.\\

Since the entropy of a geodesic flow is defined as the entropy of its time-one map, all the preceding notions will be applied to the diffeomorphism $\phi^1$. Accordingly, throughout this paper, whenever we refer to Lyapunov exponents, the Oseledets decomposition, or regular points of the geodesic flow, we mean those associated with the time-one map $\phi^1$.	
	
\section{Existence of Lyapunov exponents}
Before proving Theorem 1.1, we establish that for every $\theta \in SM$, and for every Jacobi field $J$ along $\gamma_\theta$ satisfying $ \left\| J(0) \right\|^2+ \left\| J'(0) \right\|^2=1$, both $ \left\| J(1) \right\|$ and $ \left\| J'(1) \right\|$ are bounded above by a constant depending only on the lower bound of the sectional curvature.
\begin{lemma}
	Let $M$ be a complete Riemannian manifold without conjugate points and sectional curvature bounded below by $-c^2$, where $c\ge 0$. Then, for every $\theta\in SM$ and every $\xi\in T_\theta SM$ with $\left\| \xi \right\|=1$, the Jacobi field $J_\xi$ along $\gamma_\theta$, determined by $\xi$, satisfies
	\begin{align*}
		\left\| J_\xi(1) \right\| \le C_1(c), \quad \left\| J'_\xi(1) \right\|\le C_2(c) 
	\end{align*}
	where
	\begin{align*}
		&C_1(c) = 1 + (1+2c) \frac{\sinh (c)}{c} + 2e^c,\\
		&C_2(c) =  (1+2c) \cosh (c) + 2c e^{c}.
	\end{align*}
\end{lemma}
\begin{proof}
	Fix $\theta=(x,v)\in SM$ and consider $\xi\in T_\theta SM$ with $\left\| \xi \right\|=1$. By \eqref{iso} the Jacobi field $J_\xi(t)$ along the unit geodesic $\gamma_\theta(t)$ satisfies $J_\xi(0)=d\pi_\theta(\xi)$ and $J'_\xi(0)=K_\theta(\xi)$. Write $$J_\xi = J^\top_\xi + J^\perp_\xi,$$ where $J^\top_\xi$ and $J^\perp_\xi$ denote the tangential and normal components of $J_\xi$ with respect to the geodesic $\gamma_\theta$, respectively. Suppose that $c>0$.\\
	
	\textbf{Boundedness of the tangential component $J^\top_\xi$:}\\
	The tangential component of a Jacobi field along a geodesic is explicitly given by (see \cite{lee})
	\begin{align*}
		J^\top_\xi (t)&= \left\langle J_\xi(t),\gamma'_\theta(t)\right\rangle \gamma'_\theta (t)\\
		&= \left( \left\langle J'_\xi(0),\gamma'_\theta (0)\right\rangle t + \left\langle J_\xi(0),\gamma'_\theta (0)\right\rangle \right) \gamma'_\theta(t).
	\end{align*}
	Since $\xi\in T_\theta SM$, we have that $0=\left\langle k_\theta(\xi),v\right\rangle=\left\langle J'_\xi(0),\gamma'_\theta (0)\right\rangle$. Then $$ J^\top_\xi (t)=\left\langle J_\xi(0),\gamma'_\theta (0)\right\rangle \gamma'_\theta(t)$$ and therefore for $t=1$ we obtain
	\begin{align}\label{tangential}
		\left\| J_\xi^{\top}(1)\right\|\le 1 \hspace{0.5cm} \text{and} \hspace{0.5cm}   \left\| \left(J_\xi^{\top}\right)'(1)\right\|=0.
	\end{align}
	
	\textbf{Boundedness of the normal component $J^\perp_\xi$:}\\
	We have that 
	\begin{align*}
		J_\xi^\perp (t) &= J_\xi(t) - J^\top_\xi(t) \\
		&= J_\xi(t) - \left\langle J_\xi(0),\gamma'_\theta (0)\right\rangle \gamma'_\theta(t).
	\end{align*}
	Then
	\begin{align}\label{nort}
		\left\| J_\xi^\perp (0)   \right\| \le  \left\| J_\xi (0)   \right\| +  \left\| J_\xi (0)   \right\|\le 2 \hspace{0.4cm}\text{and}\hspace{0.4cm}  \left\|   \left( J_\xi^\perp\right)' (0) \right\| =  \left\| J'_\xi (0)   \right\| \le 1.
	\end{align}
	Let $J_s$ be the normal stable Jacobi field along $\gamma_\theta$ with the same initial value as $J_\xi^\perp$, that is, $$J_s(0)=J_\xi^\perp(0).$$
	Since $S(0)=\text{id}$, such a Jacobi field is uniquely determined and is given by
	$$J_s(t)=S(t)J_\xi^\perp(0).$$ Define
	\begin{align*}
		\Tilde{J}(t)=   J_\xi^\perp(t) - J_s(t).
	\end{align*}
	Since the Jacobi equation \eqref{jacobi} is linear, $\tilde{J}$ is also a normal Jacobi field. Moreover, its initial conditions are
	\begin{align*}
		\tilde{J}(0) &= J_\xi^\perp(0) - J_s(0) = 0, \\
		\tilde{J}'(0) &= \left( J_\xi^\perp\right)'(0) - J'_s(0).
	\end{align*}
	Since $\tilde{J}(0)=0$, by Rauch's comparison theorem we have that
	\begin{align}\label{rauch}
		\left\|  \tilde{J}(1) \right\| \le  \left\|\tilde{J}'(0)\right\| \frac{\sinh (c)}{c}.
	\end{align}
	Moreover, applying Proposition 2.1, we obtain
	\begin{align}\label{derivada}
		\left\|  \tilde{J}'(1) \right\|\le \left\|\tilde{J}'(0)\right\| \cosh (c).
	\end{align}
	By definition of $\tilde{J}$, Remark 2.3 and \eqref{nort} we obtain
	\begin{align}\label{norma6}
		\nonumber
		\left\|\tilde{J}'(0)\right\| &\le \left\|(J_\xi^\perp)'(0)\right\| + \left\|J'_s(0)\right\|\\ \nonumber
		&\le 1+ c  \left\|J_s(0)\right\| \\ \nonumber
		&\le  1 + c \left\|J_\xi^\perp(0)\right\| \\ 
		&\le 1 + 2c.
	\end{align}
	Substituting the estimate in \eqref{norma6} into \eqref{rauch} and \eqref{derivada}, we obtain
	\begin{align}\label{Jtilde1}
		\left\|  \tilde{J}(1) \right\| &\le (1+2c) \frac{\sinh (c)}{c}
	\end{align}
	and
	\begin{align}\label{Jtilde2}
		\left\|  \tilde{J}'(1) \right\| &\le (1+2c) \cosh (c).
	\end{align}
	For the stable Jacobi field $J_s$, Remark 2.3 gives
	\begin{align*}
		\frac{d}{dt}\left\|J_s(t)\right\|^2 = 2\left\langle J'_s(t), J_s(t)\right\rangle \le 2\|J'_s(t)\| \|J_s(t)\| \le 2c \|J_s(t)\|^2.
	\end{align*}
	Therefore,
	\begin{align*}
		\frac{d}{dt}\left\| J_s(t)\right\|  \le c \left\|J_s(t)\right\|.   
	\end{align*}
	Integrating this inequality from $0$ to $1$ and using \eqref{nort}, we obtain
	\begin{align}\label{estable}
		\left\| J_s(1)\right\|  \le  \left\| J_s(0)\right\| e^c = \left\| J_\xi^\perp(0)\right\| e^c \le 2e^c.
	\end{align}
	Moreover, using Remark 2.3 once again, we have
	\begin{align}\label{establed}
		\left\| J'_s(1)\right\|  \le c \left\| J_s(1)\right\| \le 2c e^c.
	\end{align}
	Finally, combining the estimates \eqref{Jtilde1}, \eqref{Jtilde2}, \eqref{estable} and \eqref{establed}, we obtain
	\begin{align}\label{casi1}
		\left\| J_\xi^\perp(1)\right\| \le \left\| \tilde{J}(1)\right\|  + \left\| J_s(1)\right\| \le (1+2c) \frac{\sinh (c)}{c} + 2e^c
	\end{align}
	and
	\begin{align}\label{casi2}
		\left\| (J_\xi^\perp)'(1)\right\| \le \left\| \tilde{J}'(1)\right\|  + \left\| J'_s(1)\right\| \le (1+2c) \cosh (c) + 2ce^c.
	\end{align}
	
	\textbf{Boundedness of the Jacobi field $J_\xi$:}\\
	Define the constants
	\begin{align*}
		&C_1(c) = 1 + (1+2c) \frac{\sinh (c)}{c} + 2e^c,\\
		&C_2(c) =  (1+2c) \cosh (c) + 2c e^{c}.
	\end{align*}
	Combinig \eqref{tangential}, \eqref{casi1} and \eqref{casi2}, we obtain
	\begin{align*}
		\left\| J_\xi(1)\right\| &\le \left\| J_\xi^\top(1)\right\| + \left\| J_\xi^\perp(1)\right\|\le  C_1(c),\\
		\left\| J_\xi'(1)\right\| &\le \left\| (J_\xi^\top)'(1)\right\| + \left\| (J_\xi^\perp)'(1)\right\| \le C_2(c). 
	\end{align*}
	For $c=0$, we interpret $\dfrac{\sinh (c)}{c} = 1$ by continuity. Hence,$$C_1(0) = 1 + 1 + 2 = 4, \hspace{0.5cm}C_2(0) = 1.$$ 
	These values are consistent with the flat case. Indeed, if the curvature vanishes identically, every Jacobi field along $\gamma_\theta$ can be written as
	\begin{align*}
		J(t)=(a+bt)V(t),
	\end{align*}
	where $V(t)$ is a parallel vector field along $\gamma_\theta$. Since parallel vector fields have constant norm,
	\begin{align*}
		\left\|J(0)\right\| = \left| a \right| \left\|V(t)\right\| \hspace{0.4cm}\text{and}\hspace{0.4cm}
		\left\|J'(0)\right\| = \left| b \right| \left\|V(t)\right\|,
	\end{align*}
	for every $t\in \mathbb{R}$. Therefore, since $ \left\| J(0) \right\|^2+ \left\| J'(0) \right\|^2=1$, it follows that
	\begin{align*}
		\left\|J(1)\right\| 
		&\le \left| a \right| \left\|V(1)\right\|+\left| b \right| \left\|V(1)\right\| \\
		&= \|J(0)\|+\|J'(0)\| \\
		&\le \sqrt{2\bigl(\|J(0)\|^2+\|J'(0)\|^2\bigr)}\\
		&= \sqrt{2}.
	\end{align*}
	Moreover,
	\begin{align*}
		J'(t)=bV(t),
	\end{align*}
	and therefore
	\begin{align*}
		\left\|J'(1)\right\| = \left| b \right| \left\|V(1)\right\| = \left\|J'(0)\right\| \le 1.
	\end{align*}
	Consequently, the constants $C_1(0)$ and $C_2(0)$ provide valid bounds for the norms of Jacobi fields and their covariant derivatives at time $t=1$ in the flat case as well. This completes the proof.
\end{proof}
\begin{remark}
	Under the same assumptions as in Lemma 3.1, applying the lemma to the Jacobi field $\tilde J_\xi(t):= J_\xi(-t)$ along the reversed geodesic $\tilde{\gamma}_\theta(t):=\gamma_\theta(-t)$, for $t>0$, we obtain
	\begin{align*}
		\left\| J_\xi(-1)\right\| &\le C_1(c),\\
		\left\| J_\xi'(-1)\right\| &\le C_2(c). 
	\end{align*}
\end{remark}

\textbf{Proof of Theorem 1.1.} Fix $\theta\in SM$ and consider $\xi\in T_\theta SM$ with $\left\| \xi \right\|=1$. Recall that $$d\phi^1_{\theta}(\xi) = (J_\xi(1),J'_\xi(1)).$$ Therefore, by Lemma 3.1,
\begin{align*}
	\left\|d\phi^1_{\theta}(\xi)\right\|^2 = \left\| J_\xi(1)\right\|^2 + \left\| J_\xi'(1)\right\|^2 \le C_1(c)^2 + C_2(c)^2.
\end{align*}
Hence,
\begin{align}\label{cotasup}
	\left\|d\phi^1_{\theta}\right\|   \le \sqrt{C_1(c)^2 + C_2(c)^2}.
\end{align}
Similarly, applying Remark 3.2, we obtain
\begin{align*}
	\left\|d\phi^{-1}_{\theta}\right\|   \le \sqrt{C_1(c)^2 + C_2(c)^2}.
\end{align*}
Taking the supremum over all $\theta \in SM$, we obtain a uniform bound for $\|d\phi^{\pm 1}\|$. Consequently, $$\log^+\|d\phi^{\pm 1}\| \in L^1(\mu)$$ for every $\phi^t$-invariant probability measure $\mu$ on $SM$. This completes the proof.$\hfill\square$

\section{Ruelle's inequality}
In this section, we prove Theorem 1.2 by adapting the proof of Ruelle's inequality for Anosov geodesic flows given in \cite{cantoralpesin}. The argument does not extend directly to our setting, since the geodesic flow is no longer uniformly hyperbolic. The main additional ingredient is the application of Egorov's theorem to obtain a compact set of arbitrarily large measure on which the norm and the conorm of the derivative of the flow satisfy uniform exponential estimates governed by the extremal Lyapunov exponents. These estimates are crucial for the construction of the measurable partition. The results established below provide the necessary ingredients to carry out this extension, thereby proving Ruelle's inequality for the more general class of complete Riemannian manifolds without conjugate points.\\

Let $M$ be a complete Riemannian manifold satisfying the hypotheses of Theorem 1.2, and let $\mu$ be a $\phi^t$-invariant probability measure on $SM$. By Theorem 1.1, there exists a measurable set $\Lambda\subset SM$ with $\mu(SM\setminus\Lambda)=0$ such that, for every $\theta\in \Lambda$, the Lyapunov exponents of $\phi$ exist. More precisely, there exist real numbers $$\left\lbrace \lambda_1(\theta)<\lambda_2(\theta)<\ldots<\lambda_{l(\theta)}(\theta)\right\rbrace$$  and an invariant splitting $$T_\theta SM=\bigoplus_{i=1}^{l(\theta)}H_i(\theta),$$ such that, for every vector $\omega\in H_i(\theta)\setminus \left\lbrace 0\right\rbrace $, we have  
\begin{align*}
	\lim_{n\rightarrow \pm \infty} \dfrac{1}{n}\log \left\|d\phi^n_\theta(\omega) \right\|=\lambda_i(\theta). 
\end{align*}
For each $\theta\in \Lambda$, denote by
\begin{align*}
	E^{s}(\theta)&= \oplus_i \left\lbrace H_i(\theta): \lambda_i(\theta)< 0\right\rbrace,\\
	E^{c}(\theta)&=\oplus_i \left\lbrace H_i(\theta): \lambda_i(\theta)=0\right\rbrace,\\
	E^u(\theta)&= \oplus_i \left\lbrace H_i(\theta): \lambda_i(\theta)>0\right\rbrace.
\end{align*}
For simplicity, we assume that $\mu$ is an ergodic $\phi^t$-invariant probability measure on $SM$. In this case, we denote the Lyapunov exponents by $\left\lbrace \lambda_i\right\rbrace$  and their corresponding multiplicities by $\left\lbrace k_i\right\rbrace$. We also write $\lambda_{\min}$ and $\lambda_{\max}$ for the smallest and largest Lyapunov exponents, respectively. The proof for the non-ergodic case follows from the ergodic decomposition of such a measure. We can also assume that $\phi=\phi^1$ is an ergodic transformation with respect to $\mu$. If it is not the case, we can choose an ergodic time $t_1$ for $\mu$ and prove the theorem for the map $\phi^{t_1}$. This is possible since there exist infinitely many $t\in \mathbb{R}$ for which $\phi^t$ is ergodic (see \cite[Theorem 3.2]{Losert}). The proof of the theorem for the map $\phi^{t_1}$ implies the proof for the map $\phi$ because the entropy of $\phi^{t_1}$ and the Lyapunov exponents are both $t_1$-multiples of the respective values of $\phi$.\\

Consider $\Psi$ the velocity-reversal map defined by
\begin{align*}
	\Psi:SM\rightarrow SM, \hspace{0.2cm} \Psi(x,v)=(x,-v).
\end{align*}
We know that $\Psi$ is an isometry and satisfies
\begin{align}\label{isometria}
	\Psi\circ \phi^t = \phi^{-t}\circ \Psi, \hspace{1cm} \forall\hspace{0.1cm} t\in \mathbb{R}.
\end{align}
The following proposition shows that the map $\Psi$ sends each Oseledets subspace to the Oseledets subspace corresponding to the opposite Lyapunov exponent. Since $\mu$ is assumed to be ergodic, the Lyapunov exponents and their multiplicities are $\mu$-almost everywhere constant. Therefore, the Lyapunov spectrum is symmetric with respect to the origin, and $E^s$ and $E^u$ have the same dimension in $\Lambda$.
\begin{proposition}
	Let $\theta\in \Lambda$, and let
	\begin{align*}
		T_\theta SM=\bigoplus_{i=1}^k H_i(\theta)
	\end{align*}
	be the Oseledets decomposition at $\theta$. Then
	$\Psi(\theta)\in\Lambda$ and  
	\begin{align*}
		d\Psi_\theta\big(H_i(\theta)\big)=H_{-\lambda_i}(\Psi(\theta)),
	\end{align*}
	where $H_{-\lambda_i}(\Psi(\theta))$ denotes the Oseledets subspace associated with the exponent $-\lambda_i$. In particular,
	\begin{align*}
		\dim H_i(\theta)=\dim H_{-\lambda_i}(\Psi(\theta)).  
	\end{align*}
\end{proposition}
\begin{proof}
	Consider $\omega\in H_i(\theta)\setminus \left\lbrace 0\right\rbrace$. From \eqref{isometria} we have that
	\begin{align*}
		\left\| d\phi_\theta^{-n}(\omega)\right\|&=\left\| d\Psi_{\phi^n(\Psi(\theta))}\circ d\phi_{\Psi(\theta)}^n\circ d \Psi_\theta(\omega)\right\| \\
		&=\left\| d\phi^n_{\Psi(\theta)}(\eta)\right\|,
	\end{align*}
	where $\eta:= d \Psi_\theta(\omega)$. Since every non-zero vector in $H_i(\theta)$ has Lyapunov exponent $\lambda_i$,
	\begin{align*}
		\lim_{n\to+\infty}
		\frac1n
		\log
		\|d\phi^{-n}_\theta(\omega)\|
		=
		-\lambda_i.  
	\end{align*}
	Hence,
	\begin{align}\label{oposite}
		\lim_{n\rightarrow + \infty} \dfrac{1}{n}\log \left\|d\phi^n_{\Psi(\theta)}(\eta) \right\|= \lim_{n\rightarrow + \infty} \dfrac{1}{n}\log \left\|d\phi^{-n}_\theta(\omega) \right\|= -\lambda_i.
	\end{align}
	Similarly,
	\begin{align}\label{lya}
		\lim_{n\rightarrow + \infty} \dfrac{1}{n}\log \left\|d\phi^{-n}_{\Psi(\theta)}(\eta) \right\|= \lim_{n\rightarrow + \infty} \dfrac{1}{n}\log \left\|d\phi^{n}_\theta(\omega) \right\|= \lambda_i.   
	\end{align}
	Define
	\begin{align*}
		F_{i}(\Psi(\theta)):=d\Psi_\theta(H_i(\theta)). 
	\end{align*}
	Using \eqref{isometria} and the invariance of $H_i(\theta)$ we have that
	\begin{align*}
		d\phi_{\Psi(\theta)}(F_i (\Psi(\theta)))&=d\phi_{\Psi(\theta)}\circ d\Psi_\theta(H_i(\theta))\\
		&=d(\phi\circ \Psi)_\theta (H_i(\theta))\\
		&=d(\Psi\circ \phi^{-1})_{\theta}(H_i(\theta))\\
		&=d\Psi_{\phi^{-1}(\theta)} \circ d\phi^{-1}_\theta (H_i(\theta))\\
		&= d\Psi_{\phi^{-1}(\theta)} (H_i(\phi^{-1}(\theta)))\\
		&=F_i(\Psi(\phi^{-1}(\theta)))\\
		&=F_i(\phi(\Psi(\theta))).
	\end{align*}
	Hence, the subspaces $F_i(\Psi(\theta))$ are $d\phi$-invariant. Moreover, by \eqref{oposite} and \eqref{lya}, every non-zero vector in
	$F_i(\Psi(\theta))$ has Lyapunov exponent $-\lambda_i$. Since $d\Psi_\theta$ is an isomorphism,
	\begin{align*}
		T_{\Psi(\theta)}SM=\bigoplus_{i=1}^k F_i(\Psi(\theta)).
	\end{align*}
	Therefore, this is a $d\phi$-invariant splitting of $T_{\Psi(\theta)}SM$ with Lyapunov exponents $-\lambda_i$. Hence, $\Psi(\theta)\in\Lambda$. By the uniqueness of the Oseledets decomposition, we conclude that
	\begin{align*}
		d\Psi_\theta(H_i(\theta)) = F_i(\Psi(\theta))=H_{-\lambda_i}(\Psi(\theta))
	\end{align*}
	and
	\begin{align*}
		\dim H_i(\theta)=\dim H_{-\lambda_i}(\Psi(\theta)).    
	\end{align*}
	This completes the proof.
\end{proof}
Fix $\varepsilon>0$. For each $\theta\in \Lambda$ and $n\in \mathbb{N}$, denote by $\displaystyle \left\| d\phi^n_\theta\right\|^*=\inf_{\left\|\omega \right\| =1}\left\|d\phi^n_\theta(\omega) \right\|$ and define the following sequence of functions
\begin{align*}
	&f_n(\theta):=\dfrac{1}{n}\log\left\|d\phi^n_\theta\right\|, \\
	& g_n(\theta):=\dfrac{1}{n}\log  \left\|d\phi^n_\theta\right\|^* .
\end{align*}
By Furstenberg-Kesten theorem (see \cite{FK}) we have that
\begin{align}\label{maxima}
	\lim_{n\rightarrow+\infty}f_n(\theta)=\lim_{n\rightarrow+\infty}\dfrac{1}{n}\log\left\|d\phi^n_\theta\right\|=\lambda_{\max} ,\hspace{1cm} \forall\hspace{0.1cm} \theta\in \Lambda.
\end{align}
Similarly, for each $\theta \in \Lambda$, let $H_1(\theta)$ be  the Oseledets subspace corresponding to the smallest Lyapunov exponent $\lambda_{\min}$. Every non-zero vector $\omega\in T_\theta SM$ satisfies
\begin{align*}
	\lim_{n\rightarrow+\infty}\dfrac{1}{n}\log\left\|d\phi^n_\theta(\omega)\right\|=\lambda_i,
\end{align*}
where $\lambda_i$ is the largest Lyapunov exponent corresponding to a non-zero Oseledets component of $\omega$. Consequently, vectors in $H_1(\theta)$ realize the smallest possible exponential growth, while every vector having no component in $H_1(\theta)$ grows at a strictly larger exponential rate. It follows that
\begin{align}\label{minima2}
	\lim_{n\rightarrow+\infty}g_n(\theta)=\lim_{n\rightarrow+\infty}\dfrac{1}{n}\log\left\|d\phi^n_\theta\right\|^*=\lambda_{\min} ,\hspace{1cm} \forall\hspace{0.1cm} \theta\in \Lambda.
\end{align}
Applying Egorov's theorem to the convergences $\eqref{maxima}$ and \eqref{minima2}, there exists a compact set $K\subset \Lambda$ with $\mu(K)\ge 1-\varepsilon$ such that the splitting $$T_\theta SM=E^{s}(\theta)\oplus E^c(\theta)\oplus E^{u}(\theta)$$ is continuous when $\theta$ varies in $K$, and there exists $n_0\in \mathbb{N}$ such that
\begin{align}\label{comparacion}
	e^{n(\lambda_{\min} - \varepsilon)} < \left\|d\phi^n_{\theta}\right\|^*\le \left\|d\phi^n_{\theta}\right\|  < e^{n(\lambda_{\max}+\varepsilon)} 
\end{align}
for every $\theta\in K$ and every $n\ge n_0$.\\

We now proceed to establish several auxiliary results which will play a crucial role in the proof of Ruelle's inequality. We begin by showing that estimate \eqref{comparacion} yields a uniform comparison between $\left\| d\phi^m \right\|$ and $\left\| d\phi^m \right\|^*$ on the set $K$.
\begin{proposition}
	For large enough $m\in \mathbb{N}$, there exists $\kappa>1$, depending on $m$, such that
	\[
	\left\| d\phi^m_\theta \right\| \le \kappa \left\| d\phi^m_\theta \right\|^*
	\]
	for every $\theta\in K$.
\end{proposition}
\begin{proof}
	For $m\ge n_0$, it follows from \eqref{comparacion} that
	\begin{align*}
		\|d\phi^m_\theta\| \le e^{m(\lambda_{\max}-\lambda_{\min}+2\varepsilon)} \|d\phi^m_\theta\|^*
	\end{align*}
	for every $\theta\in K$. By setting $\kappa = e^{m(\lambda_{\max}-\lambda_{\min}+2\varepsilon)}>1$ the proposition follows.
\end{proof}
We now establish a global estimate, valid on the unit tangent bundle $SM$, showing that $\left\| d\phi^m_\theta\right\|$ is uniformly bounded above and below. 
\begin{proposition}
	For every $m\in \mathbb{N}$, there exist $P_1, P_2>0$, with $P_1$ depending on $m$, such that
	\begin{align*}
		P_2<\left\| d\phi^m_\theta\right\| <P_1,
	\end{align*}
	for every $\theta\in SM$.
\end{proposition}
\begin{proof} Since the geodesic vector field is invariant under the flow,
	\begin{align*}
		\left\| d\phi^m_\theta\right\|\ge 1.
	\end{align*}
	Therefore, we may take $P_2=\dfrac{1}{2}$. On the other hand, by \eqref{cotasup},
	\begin{align}\label{betanov}
		\left\|d\phi^1_\theta\right\| < \sqrt{C_1(c)^2 + C_2(c)^2}+1:=h(c)
	\end{align}
	for every $\theta\in SM$. Then, we can set $P_1=h(c)^m$, which completes the proof.
\end{proof}
A direct consequence of Proposition 4.3 is the following result.
\begin{corollary}
	There exists $\tau\in (0,1)$, depending on $m$, such that
	\begin{align*}
		\tau \left\| d\phi^m_{\tilde{\theta}}\right\|< \left\| d\phi^m_\theta\right\|,
	\end{align*}
	for every $\theta, \tilde{\theta}\in SM$.
\end{corollary}
\begin{proof}
	By Proposition 4.3 we have that
	\begin{align*}
		\dfrac{P_2}{P_1}<\dfrac{\left\| d\phi^m_\theta\right\|}{\left\| d\phi^m_{\tilde{\theta}}\right\|} <\dfrac{P_1}{P_2}
	\end{align*}
	Considering $\tau=\dfrac{P_2}{P_1}=\dfrac{1}{2h(c)^m}$ the conclusion of the corollary follows. 
\end{proof} 
Fix $m\in \mathbb{N}$ sufficiently large. The estimates established above allow us to obtain the following theorem, which is the analogue of Theorem 5.1 in \cite{cantoralpesin}. Indeed, once Propositions 4.2 and 4.3 and Corollary 4.4 are available, the proof follows exactly the same arguments as in \cite{cantoralpesin}. For the convenience of the reader, we state the result here and omit the proof.
\begin{theorem}
	Let $M$ be a complete Riemannian manifold without conjugate points and sectional curvature bounded below by $-c^2$, for some $c>0$. Then, there exists $\varrho:=\varrho(K)\in (0,1)$ such that
	\begin{align*}
		\phi^m(exp_\theta(B(0,\tau\kappa^{-1}\varrho)))\subseteq exp_{\phi^m(\theta)} (d \phi^m_\theta(B(0,\varrho)),
	\end{align*}
	for every $\theta\in K$.
\end{theorem}
Denote by $\varrho_m=\tau \kappa^{-1}\varrho<1$, where the constants $\tau, \kappa$ and $\varrho$ come from Corollary 4.4, Proposition 4.2 and Theorem 4.5, respectively. Using techniques from separate sets, as applied in \cite{cantoralpesin}, we define a finite partition $\mathcal{P}=\mathcal{P}_{K}\cup \left\lbrace SM\setminus K\right\rbrace $ of $SM$ as follows:
\begin{itemize}
	\item[.] $\mathcal{P}_{K}$ is a partition of $K$ such that for every $X\in \mathcal{P}_{K}$, there exist balls $B(x,r')$ and $B(x,r)$ such that the constants satisfy $0<r'<r<2r'\le   \dfrac{\varrho_m}{2}$ and $$B(x,r')\subset X\subset B(x,r).$$
	\item[.] There exists a constant $\zeta>0$ such that the cardinal of $\mathcal{P}_{K}$, denoted by $\left| \mathcal{P}_{K}\right|$, satisfies $$\left| \mathcal{P}_{K}\right|\le \zeta\cdot(\varrho_m)^{-\dim(SM)}.$$
	\item[.] $h_{\mu}(\phi^m,\mathcal{P})\ge h_{\mu}(\phi^m)-\varepsilon$. 
\end{itemize}
By definition of entropy,
\begin{align}\label{ruelle}
	h_{\mu}(\phi^m,\mathcal{P})&=\lim_{k\to +\infty} H_{\mu}\left( \left. \mathcal{P}\right|\phi^m\mathcal{P} \vee\ldots \vee \phi^{km}\mathcal{P} \right) \nonumber \\
	&\le H_{\mu}\left( \left. \mathcal{P}\right| \phi^m\mathcal{P}\right) \nonumber \\
	&\le \sum_{D\in \phi^m\mathcal{P}}\mu(D)\cdot\log\text{card}\left\lbrace X\in \mathcal{P}: X\cap D\neq\emptyset \right\rbrace.
\end{align}
Denote by $\varphi=\sup_{\theta\in SM}\left\|d\phi_\theta \right\|>1$. First, we estimate the number of elements $X\in \mathcal{P}$ that intersect a given element $D\in \phi^m\mathcal{P}$.
\begin{lemma}
	There exists a constant $L_1>0$ such that if $D\in \phi^m\mathcal{P}$ then
	\begin{align*}
		\emph{card} \left\lbrace X\in \mathcal{P}: X\cap D\neq\emptyset \right\rbrace \le L_1\cdot\max\left\lbrace \varphi^{m\cdot\dim (SM)}, (\varrho_m)^{-\dim(SM)} \right\rbrace .
	\end{align*}
\end{lemma}
\begin{proof}
	Consider $D\in \phi^m\mathcal{P}$, then $D=\phi^m(X')$ for some $X'\in \mathcal{P}$.\\
	Case I: $X'\in \mathcal{P}_{K}$.\\
	By the mean value inequality
	\begin{align*}
		\text{diam}(D)&=\text{diam}(\phi^m(X'))\\
		&\le \sup_{\theta\in SM} \left\|d\phi_\theta \right\|^{m}\cdot\text{diam}(X')\\
		&\le \varphi^m\cdot 4r',
	\end{align*}
	since $X'\subset B(x,2r')$. If $X\in \mathcal{P}_{K}$ satisfies $X\cap D\neq \emptyset$, then $X$ is contained in a $4r'$-neighborhood of $D$, denoted by $W$. Since $\varphi^m>1$ we have that
	\begin{align*}
		\text{diam}(W)& \le \varphi^m\cdot 4r' + 8r'\\
		&=4r'\cdot \left( \varphi^m + 2\right)\\
		&<12r'\cdot\varphi^m.
	\end{align*} 
	Hence
	\begin{align}\label{vol1}
		\sum_{\left\lbrace X\in \mathcal{P}_{K}:X\cap D\neq \emptyset\right\rbrace }\text{vol}(X)\le \text{vol}(W)\le A_1 \cdot (r')^{\dim (SM)} \cdot\varphi^{m\cdot\dim (SM)},
	\end{align}
	where $A_1>0$. Since $X\in \mathcal{P}_{K}$ contains a ball of radius $r'$, the volume of $X$ is bounded below by
	\begin{align}\label{vol2}
		A_2\cdot (r')^{\dim (SM)}\le\text{vol}(X),
	\end{align}
	where $A_2>0$. From \eqref{vol1} and \eqref{vol2} we have that
	\begin{align*}
		\text{card}\left\lbrace X\in \mathcal{P}:X\cap D\neq \emptyset\right\rbrace &\le \dfrac{A_1}{A_2} \cdot\varphi^{m\cdot\dim (SM)} +1\\
		&\le \left(\dfrac{A_1}{A_2}+1\right) \cdot\varphi^{m\cdot\dim (SM)} .
	\end{align*}
	\noindent
	Case II: $X'=SM\setminus K$.\\
	In this case, we have that 
	\begin{align*}
		\text{card}\left\lbrace X\in \mathcal{P}:X\cap D\neq \emptyset\right\rbrace&\le \left| \mathcal{P}_{K}\right| +1\\
		&\le (\zeta+1)(\varrho_m)^{-\dim(SM)}.
	\end{align*}
	Considering $L_1=\max\left\lbrace \dfrac{A_1}{A_2}+1, \zeta+1\right\rbrace$ we obtain the desired result.
\end{proof}
Let $\Lambda_m\subset \Lambda$ be the set of regular points $\theta\in SM$ satisfying
\begin{align*}
	e^{k\left( \lambda(\theta,\xi)-\varepsilon\right) } \left\| \xi\right\| \le \left\| d\phi^k_\theta (\xi) \right\|\le e^{k\left( \lambda(\theta,\xi)+\varepsilon\right) } \left\| \xi\right\|,
\end{align*}
for every $k\ge m$ and every $\xi\in T_\theta SM$, where $$	\lambda(\theta,\xi)=\displaystyle \lim_{n\rightarrow + \infty} \dfrac{1}{n}\log \left\|d \phi^n_\theta(\xi) \right\|.$$ 
The set $\Lambda_m$ consists of regular points for which the growth of the derivative is controlled, up to an error $\varepsilon$, for all times $k\ge m$. We next derive a sharper exponential estimate for the number of elements
$D\in \phi^m\mathcal{P}_K$ that intersect the set $\Lambda_m$.
\begin{lemma}
	If $D\in \phi^m\mathcal{P}_K$ intersects $\Lambda_m$, then there is a constant $L_2>0$ such that
	\begin{align*}
		\emph{card}\left\lbrace X\in \mathcal{P}: X\cap D\neq \emptyset\right\rbrace\le L_2\cdot e^{m\varepsilon}\prod_{i:\lambda_i>0}e^{m(\lambda_i+\varepsilon)k_i}. 
	\end{align*}
\end{lemma}
\begin{proof}
	Let $X'\in \mathcal{P}_{K}$ be such that $\phi^m(X')=D$ and assume that $X'\cap \Lambda_m\neq \emptyset$. Choose a point $\theta\in X'\cap \Lambda_m$, and consider the ball $B=B(0,\varrho)\subset T_\theta SM$. We first claim that $$X'\subseteq exp_\theta(B(0,\varrho_m)),$$
	where $exp_\theta$ denotes the exponential map defined on $T_\theta SM$. Indeed, let $z\in X'$. Since $SM$ is complete with respect to the Sasaki metric (see \cite[Lemma 2.1]{cantoralpesin}) we can choose $w\in T_\theta SM$ such that $\gamma(t)=exp_\theta(tw)$, where $\gamma$ is a geodesic with $\gamma(0)=\theta$ and $\gamma(1)=exp_\theta(w)=z$. As diam $\mathcal{P}_{K} < \varrho_m$ then $$d(\theta,z)=l(\gamma)< \varrho_m.$$
	By Gauss's Lemma (see \cite{lee}) we obtain
	\begin{align*}
		\varrho_m&> \int_0^1\left\| \gamma'(s)\right\| ds\\
		&= \int_0^1 \left\langle d(exp_\theta)_{sw}(w),d(exp_\theta)_{sw}(w) \right\rangle^{1/2}ds \\
		&=\int_0^1 \dfrac{1}{\sqrt{s}}\left\langle d(exp_\theta)_{sw}(sw),d(exp_\theta)_{sw}(w) \right\rangle^{1/2} ds\\
		&= \int_0^1\dfrac{1}{\sqrt{s}}\left\langle sw,w\right\rangle^{1/2} ds\\
		&=\left\| w\right\| .
	\end{align*}
	Therefore, $w\in B(0,\varrho_m)$ and consequently
	$$z=exp_{\theta}(w)\in exp_\theta(B(0,\varrho_m)).$$
	Since $z\in X'$ was arbitrary, the claim follows. Therefore, by Theorem 4.5 we obtain
	\begin{align*}
		D=\phi^m(X')\subseteq B_0:=exp_{\phi^m(\theta)}(\tilde{B}_0),
	\end{align*}
	where $\tilde{B}_0=d\phi^m_\theta(B)$ is an ellipsoid. Since the curvature tensor of $M$ and its derivative are both uniformly bounded, the sectional curvature of $SM$ with respect to the Sasaki metric is also uniformly bounded (see \cite[Appendix A]{cantoralpesin}). Consequently, the curvature tensor of $SM$ is uniformly bounded as well. Applying Proposition 2.4 to the manifold $SM$, there exists $t_0>0$ such that  
	\begin{align}\label{novot0}
		\left\|d (exp_{\phi^m(\theta)})_{tv} \right\|\le \dfrac{5}{2}  ,  
	\end{align} 
	for every $\left| t\right|\le t_0$ and every $v\in T_{\phi^m(\theta)}SM$ with $\left\|v \right\|=1$. Then, for $m$ large enough, we have that
	\begin{align*}
		\text{diam} (D)&\le h(c)^m\cdot \text{diam} (X')\\
		&< h(c)^m\cdot \varrho_m\\
		&= h(c)^m\cdot\dfrac{1}{2h(c)^m}\cdot e^{-m(\lambda_{\max}-\lambda_{\min}+2\varepsilon)}\cdot \varrho\\
		&=\dfrac{1}{2}\cdot e^{-m(\lambda_{\max}-\lambda_{\min}+2\varepsilon)}\cdot \varrho\\
		&<\dfrac{t_0}{2}, 
	\end{align*}
	where $h(c)$ is the constant defined in \eqref{betanov}. Therefore, we can choose $B_0$ such that $D\subset B_0$ and diam$(B_0)<t_0$. Since diam$\mathcal{P}_{K}< \varrho_m<\varrho$, every element $X\in \mathcal{P}_{K}$ intersecting $D$ is contained in the set
	$$B_1=\left\lbrace \psi\in SM: d(\psi,B_0)<\varrho\right\rbrace.$$ 
	Since $X\subset B(x,r)$ and $2r<\varrho_m<\varrho$, then $B(x,\varrho/2)\subset B_1$. Hence, there exists a constant $b>0$ such that 
	\begin{align}\label{cardi}
		\text{card}\left\lbrace X\in \mathcal{P}_{K}: X\cap D\neq \emptyset \right\rbrace \le b\cdot\text{vol}(B_1)\cdot \varrho^{-\dim(SM)},
	\end{align}
	where vol$(B_1)$ denotes the volume of $B_1$ induced by the Sasaki metric. Consider a subset $\tilde{B}^*_0\subset \tilde{B}_0$ such that $exp_{\phi^m(\theta)}$ is a diffeomorphism between $\tilde{B}^*_0$ and $B_0$. Since
	\begin{align*}
		\left| \det d (exp_{\phi^m(\theta)})_v\right| \le \left\|d (exp_{\phi^m(\theta)})_v \right\|^{\dim(SM)}
	\end{align*}
	for every $v\in \tilde{B}^*_0$, by \eqref{novot0} we have that
	\begin{align*}
		\text{vol}(B_0)\le \left( \dfrac{5}{2}\right) ^{\dim(SM)}\cdot \text{vol}(\tilde{B}_0).
	\end{align*}
	The previous estimate shows that the volume of $B_1$ is bounded, up to a bounded factor, by the product of the lengths of the axes of the ellipsoid $\tilde{B}_0=d\phi^m_\theta(B)$. These axes are determined by the expansion rates of $d\phi^m_\theta$ along the Oseledets subspaces. For directions corresponding to non-positive Lyapunov exponents, the lengths of the axes grow at most sub-exponentially. On the other hand, for each positive Lyapunov exponent $\lambda_i$, the corresponding axes have length bounded by $e^{m(\lambda_i+\varepsilon)}$, up to a bounded factor, for all sufficiently large $m$. Therefore, 
	\begin{align*}
		\text{vol}(B_1)&\le A\cdot e^{m\varepsilon}\cdot(\text{diam}(B))^{\dim (SM)}\prod_{i:\lambda_i>0}e^{m(\lambda_i+\varepsilon)k_i}\\
		&\le A\cdot e^{m\varepsilon}\cdot(2\varrho)^{\dim (SM)}\prod_{i:\lambda_i>0}e^{m(\lambda_i+\varepsilon)k_i}\\
		& =\tilde{A}\cdot e^{m\varepsilon}\cdot \varrho^{\dim (SM)}\prod_{i:\lambda_i>0}e^{m(\lambda_i+\varepsilon)k_i},
	\end{align*}
	where $\tilde{A}=A\cdot 2^{\dim (SM)}$, for some $A>0$. Substituting this estimate into \eqref{cardi} we obtain
	\begin{align*}
		\text{card}\left\lbrace X\in \mathcal{P}: X\cap D\neq \emptyset \right\rbrace &\le b\cdot\text{vol}(B_1)\cdot \varrho^{-\dim (SM)}+1\\
		&\le b\cdot\tilde{A}\cdot e^{m\varepsilon}\prod_{i:\mathcal{X}_i>0}e^{m(\mathcal{X}_i+\varepsilon)k_i}+1\\
		&\le (b\cdot\tilde{A}+1)\cdot e^{m\varepsilon}\prod_{i:\mathcal{X}_i>0}e^{m(\mathcal{X}_i+\varepsilon)k_i}.
	\end{align*}
	Taking $L_2=b\cdot\tilde{A}+1$, the desired estimate follows.
\end{proof}
\textbf{Proof of Theorem 1.2.} From the construction of the compact set $K$, we have that $\mu(SM\setminus K)<\varepsilon$. For a fixed sufficiently large $m$, we divide the elements of the partition $\phi^m\mathcal{P}_K$
into two sets:
\begin{align*}
	\mathcal{O}_1&=\left\lbrace D\in \phi^m \mathcal{P}_K: D\cap \Lambda_m=\emptyset\right\rbrace,\\
	\mathcal{O}_2&=\left\lbrace D\in \phi^m \mathcal{P}_K: D\cap \Lambda_m \neq \emptyset \right\rbrace.
\end{align*}
By \eqref{ruelle}, Lemmas 4.6 and 4.7 we obtain
\begin{align}\label{fin}	
	mh_{\mu}(\phi)-\varepsilon&=h_{\mu}(\phi^m)-\varepsilon  \nonumber \\
	&\le h_{\mu}(\phi^m,\mathcal{P}) \nonumber \\
	&\le \sum_{D\in \phi^m\mathcal{P}}\mu(D)\cdot\log\text{card}\left\lbrace X\in \mathcal{P}: X\cap D\neq\emptyset \right\rbrace \nonumber \\
	&\le \sum_{D\in \mathcal{O}_1}\mu(D)\cdot\log\text{card}\left\lbrace X\in \mathcal{P}: X\cap D\neq\emptyset \right\rbrace \nonumber\\
	&\hspace{0.5cm} +	\sum_{D\in \mathcal{O}_2}\mu(D)\cdot\log\text{card}\left\lbrace X\in \mathcal{P}: X\cap D\neq\emptyset \right\rbrace \nonumber \\
	&\hspace{0.5cm}+\mu(\phi^m(SM\setminus K))\cdot\log\text{card}\left\lbrace X\in \mathcal{P}: X\cap \phi^m(SM\setminus K)\neq\emptyset \right\rbrace \nonumber\\ 
	&\le \sum_{D\in \mathcal{O}_1}\mu(D)\left( \log (L_1) +\dim (SM)\cdot\max\left\lbrace m\log(\varphi),-\log(\varrho_m) \right\rbrace\right) \nonumber \\
	&\hspace{0.5cm}+ \sum_{D\in \mathcal{O}_2}\mu(D)\left( \log(L_2) +m\varepsilon  +m\sum_{i:\lambda_i>0}(\lambda_i+\varepsilon)k_i\right)  \nonumber \\
	& \hspace{0.5cm} + \mu(SM\setminus K)\cdot\left( \log (L_1) +\dim (SM)\cdot\max\left\lbrace m\log(\varphi),-\log(\varrho_m) \right\rbrace\right)  \nonumber\\
	&\le \left( \log (L_1) +\dim (SM)\cdot\max\left\lbrace m\log\left(\varphi\right) ,-\log(\varrho_m) \right\rbrace\right)\cdot \mu(SM\setminus \Lambda_m) \nonumber\\
	&\hspace{0.5cm} + \log(L_2) +m\varepsilon  +m\sum_{i:\lambda_i>0}(\lambda_i+\varepsilon)k_i \nonumber\\
	&\hspace{0.5cm}+ \varepsilon\cdot\left( \log (L_1) +\dim (SM)\cdot\max\left\lbrace m\log\left(\varphi\right) ,-\log(\varrho_m) \right\rbrace\right).
\end{align}
By Oseledets theorem we have that $\mu(SM\setminus \Lambda_m)\rightarrow 0$ as $m\rightarrow \infty$. Moreover, by the definition of $\varrho_m$, we have
\begin{align*}
	\lim_{m\rightarrow +\infty}\dfrac{1}{m}\log(\varrho_m) = -\log(h(c))-(\lambda_{\max}-\lambda_{\min}+2\varepsilon),
\end{align*}
where $h(c)$ is the constant defined in \eqref{betanov}. Then, dividing by $m$ in \eqref{fin} and taking $m\rightarrow +\infty$ we obtain
\begin{align*}
	h_{\mu}(\phi)\le \varepsilon + \sum_{i:\lambda_i>0}(\lambda_i+\varepsilon)k_i +\varepsilon\cdot\dim (SM)\cdot\max\left\lbrace \log\left(\varphi\right) ,\log(h(c))+\lambda_{\max}-\lambda_{\min}+2\varepsilon \right\rbrace.
\end{align*}
Since all the terms multiplied by $\varepsilon$ vanish in the limit, letting $\varepsilon\rightarrow 0$ we have
\begin{align*}
	h_{\mu}(\phi)\le  \sum_{i:\lambda_i>0}\lambda_i k_i,
\end{align*}
which is the desired upper bound. $\hfill\square$

\section{The Pesin formula}
In this section, we explain how the proof of the Pesin entropy formula for Anosov geodesic flows given in \cite{cantoralpesin} extends to the setting of manifolds without conjugate points. Rather than reproducing the proof, we verify that each of the technical ingredients used in the Anosov setting remains valid under our geometric assumptions. Once these ingredients are established, the argument of \cite{cantoralpesin} carries over with only minor modifications.\\

Let $M$ be a complete finite-volume Riemannian manifold without conjugate points and $\mu$ an $\phi^t$-invariant probability measure on $SM$ satisfying the assumptions of Theorem 1.3. By Theorem 1.1, there exists a full $\mu$-measure set $\Lambda\subset SM$ such that the Lyapunov exponents of the geodesic flow exist at every point $\theta\in \Lambda$. Following the notation introduced in the previous sections, we write
$$E^{cs}(\theta)=E^c(\theta) \oplus E^s(\theta), \hspace{1cm} \forall\hspace{0.1cm}\theta\in \Lambda.$$
Unlike the Anosov setting, where the unstable bundle has constant dimension equal to $n-1$, the dimension of the unstable Oseledets subspace $E^u(\theta)$  may vary from point to point. This is the main additional feature that must be taken into account in the proof. To overcome this difficulty, we decompose the measure according to the dimension of the unstable subspace.\\

For each $j\ge 0$, let $$\Sigma_j=\left\lbrace \theta:\dim E^u(\theta)=j\right\rbrace,$$ and define $$S=\left\lbrace j\ge 0: \mu(\Sigma_j)>0\right\rbrace.$$ Since the map $\theta\rightarrow \dim E^u(\theta)$ is $\phi$-invariant, each set $\Sigma_j$ is $\phi$-invariant. Consequently, for every $j\in S$, the probability measure
$$\mu_j(A):=\dfrac{\mu(A\cap \Sigma_j)}{\mu(\Sigma_j)}$$ is $\phi$-invariant. Moreover,
\begin{align*}
	\mu=\sum_{j\in S}\mu(\Sigma_j)\mu_j,
\end{align*}
and, by the affinity of entropy,
\begin{align*}
	h_\mu(\phi)=\sum_{j\in S} \mu(\Sigma_j)h_{\mu_j}(\phi). 
\end{align*}
It therefore suffices to prove that $$h_{\mu_j}(\phi)\ge \int_{SM} \mathcal{X}^+d\mu_j,$$ where $$\mathcal{X}^+(\theta)=\sum_{\lambda_i(\theta)>0}\lambda_i(\theta)\cdot\dim (H_i(\theta)).$$ This inequality is trivial when $j=0$, so we may assume that $j>0$. To simplify the notation, we write $\mu=\mu_j$.\\

Fix any $\varepsilon>0$. By Egorov's theorem, there exists a compact set $K\subset \Lambda$ with $\mu(K)\ge 1-\varepsilon$ such that the splitting $$T_\theta SM=E^{cs}(\theta)\oplus E^u(\theta)$$ depends continuously on $\theta\in K$. Moreover, there exist $N>0$ and constants $\alpha>\beta>1$ such that, if $g=\phi^N$, the inequalities
\begin{align*}
	\left\|d g^n_\theta(\eta) \right\|&\ge \alpha^n\left\| \eta\right\| \nonumber \\ 
	\left\|\left.d g^n_\theta\right|_{E^{cs}(\theta)}  \right\|&\le \beta^n \nonumber\\ 
	\log\left| \det\left( \left. d g^n_\theta\right|_{E^u(\theta)} \right) \right| &\ge Nn\left(\mathcal{X}^+(\theta)-\varepsilon\right)
\end{align*}
hold for all $\theta\in K$, $n\ge 0$ and $\eta\in E^u(\theta)$.\\

These are precisely the estimates required in the proof of \cite{cantoralpesin}. Since that proof relies only on the above estimates, the curvature assumptions, and Proposition 2.4, the same argument applies with no essential modifications in the present setting. Consequently, the Pesin entropy formula holds for geodesic flows on complete finite-volume manifolds without conjugate points.

	\bibliographystyle{abbrv}
	\bibliography{references}

@article{cantoralpesin,
  author = {Alexander Cantoral and Sergio Roma{\~n}a},
  title = {Ruelle's inequality and {P}esin's formula for {A}nosov geodesic flows in non-compact manifolds},
  journal = {Discrete and Continuous Dynamical Systems},
  volume = {46},
  number = {1},
  pages = {387--412},
  year = {2026},
  doi = {10.3934/dcds.2025104},
  mrclass = {37D40 (53C20)},
}

@book {paternain,
	AUTHOR = {Paternain, Gabriel P.},
	TITLE = {Geodesic flows},
	SERIES = {Progress in Mathematics},
	VOLUME = {180},
	PUBLISHER = {Birkh\"{a}user Boston, Inc., Boston, MA},
	YEAR = {1999},
	PAGES = {xiv+149},
	ISBN = {0-8176-4144-0},
	MRCLASS = {53D25 (37D40 37J99)},
	MRNUMBER = {1712465},
	MRREVIEWER = {Boris Hasselblatt},
	DOI = {10.1007/978-1-4612-1600-1},
	URL = {https://doi.org/10.1007/978-1-4612-1600-1},
}

@article {klin,
	AUTHOR = {Klingenberg, Wilhelm},
	TITLE = {Riemannian manifolds with geodesic flow of {A}nosov type},
	JOURNAL = {Ann. of Math. (2)},
	FJOURNAL = {Annals of Mathematics. Second Series},
	VOLUME = {99},
	YEAR = {1974},
	PAGES = {1--13},
	ISSN = {0003-486X},
	MRCLASS = {58E10 (58F15)},
	MRNUMBER = {377980},
	MRREVIEWER = {Robert Roussarie},
	DOI = {10.2307/1971011},
	URL = {https://doi.org/10.2307/1971011},
}

@article{FK,
  author    = {Harry Furstenberg and Harry Kesten},
  title     = {Products of Random Matrices},
  journal   = {The Annals of Mathematical Statistics},
  volume    = {31},
  number    = {2},
  pages     = {457--469},
  year      = {1960},
  month     = jun,
  doi       = {10.1214/aoms/1177705909}
}

@article {mane2,
	AUTHOR = {Ma\~{n}\'{e}, R.},
	TITLE = {On a theorem of {K}lingenberg},
	JOURNAL = {Dynamical systems and bifurcation theory ({R}io de {J}aneiro,
	1985)},
	SERIES = {Pitman Res. Notes Math. Ser.},
	VOLUME = {160},
	PAGES = {319--345},
	PUBLISHER = {Longman Sci. Tech., Harlow},
	YEAR = {1987},
	MRCLASS = {58F17 (53C22 58F15)},
	MRNUMBER = {907897},
	MRREVIEWER = {Victor Bangert},
}

@incollection{Losert,
  author    = {Viktor Losert and Klaus Schmidt},
  title     = {A Class of Probability Measures on Groups Arising from Some Problems in Ergodic Theory},
  booktitle = {Probability Measures on Groups},
  editor    = {Herbert Heyer},
  series    = {Lecture Notes in Mathematics},
  volume    = {706},
  pages     = {220--238},
  publisher = {Springer},
  address   = {Berlin, Heidelberg},
  year      = {1979},
  doi       = {10.1007/BFb0063108_20}
}

@article{nocon,
  title={{R}iemannian manifolds with {A}nosov geodesic flow do not have conjugate points},
  author={Melo, {\'I}talo and Roma{\~n}a, Sergio},
  journal={arXiv preprint arXiv:2008.12898},
  year={2020}
}

@incollection {Knieper,
	AUTHOR = {Knieper, Gerhard},
	TITLE = {Hyperbolic dynamics and {R}iemannian geometry},
	BOOKTITLE = {Handbook of dynamical systems, {V}ol. 1{A}},
	PAGES = {453--545},
	PUBLISHER = {North-Holland, Amsterdam},
	YEAR = {2002},
	MRCLASS = {37D25 (37A35 53C24)},
	MRNUMBER = {1928523},
	MRREVIEWER = {Rafael Oswaldo Ruggiero},
	DOI = {10.1016/S1874-575X(02)80008-X},
	URL = {https://doi.org/10.1016/S1874-575X(02)80008-X},
}

@article {oseledec,
	AUTHOR = {Oseledec, V. I.},
	TITLE = {A multiplicative ergodic theorem. {C}haracteristic {L}japunov,
	exponents of dynamical systems},
	JOURNAL = {Trudy Moskov. Mat. Ob\v{s}\v{c}.},
	FJOURNAL = {Trudy Moskovskogo Matemati\v{c}eskogo Ob\v{s}\v{c}estva},
	VOLUME = {19},
	YEAR = {1968},
	PAGES = {179--210},
	ISSN = {0134-8663},
	MRCLASS = {28.70 (34.00)},
	MRNUMBER = {0240280},
	MRREVIEWER = {J\'{o}zsef Sz\"{u}cs},
}

@article {riquelme,
	AUTHOR = {Riquelme, Felipe},
	TITLE = {Ruelle's inequality in negative curvature},
	JOURNAL = {Discrete Contin. Dyn. Syst.},
	FJOURNAL = {Discrete and Continuous Dynamical Systems. Series A},
	VOLUME = {38},
	YEAR = {2018},
	NUMBER = {6},
	PAGES = {2809--2825},
	ISSN = {1078-0947,1553-5231},
	MRCLASS = {37D25 (28D20 37A35 37D10 37D40)},
	MRNUMBER = {3809061},
	MRREVIEWER = {Alejandro\ Mario\ Mes\'{o}n},
	DOI = {10.3934/dcds.2018119},
	URL = {https://doi.org/10.3934/dcds.2018119},
}

@article{ruelle,
	title={An inequality for the entropy of differentiable maps},
	author={Ruelle, David},
	journal={Boletim da Sociedade Brasileira de Matem{\'a}tica-Bulletin/Brazilian Mathematical Society},
	volume={9},
	number={1},
	pages={83--87},
	year={1978},
	publisher={Springer-Verlag Berlin/Heidelberg}
}

@book{lee,
	title={Introduction to Riemannian manifolds},
	author={Lee, John M},
	volume={2},
	year={2018},
	publisher={Springer}
}
\end{document}